\documentclass{article}
\usepackage{mathrsfs}
\usepackage{epsfig, graphicx}
\usepackage{latexsym,amsfonts,amsbsy,amssymb}
\usepackage{amsmath,amsthm}
\usepackage{hyperref}
\usepackage{color}
\makeatletter 

\@addtoreset{equation}{section}
\makeatother 
\allowdisplaybreaks 

\newtheorem{theorem}{Theorem}[section]
\newtheorem{lemma}{Lemma}[section]

\newtheorem{remark}{Remark}[section]
\newtheorem{algorithm}{Algorithm}[section]
\newtheorem{definition}{Definition}[section]
\newtheorem{proposition}{Proposition}[section]

\begin{document}
\title{Energy Error Estimates of Subspace Projection Method and Multigrid Algorithms for Eigenvalue
Problems\footnote{This work was supported in part by
National Natural Science Foundations of China (NSFC 91330202,
11371026, 11001259, 11031006, 2011CB309703), Science Challenge
Project (No. JCKY2016212A502), the National Center for Mathematics
and Interdisciplinary Science, CAS.}}
\author{Yunhui He\footnote{Department of Mathematics and Statistics,
Memorial University of Newfoundland, St. John's, NL A1C 5S7, Canada (yunhui.he@mun.ca)},\ \
Qichen Hong\footnote{LSEC, ICMSEC, Academy of Mathematics and Systems Science,
Chinese Academy of Sciences, Beijing 100190, P.R. China, and School of Mathematical Sciences, University
of Chinese Academy of Sciences, Beijing, 100049, China (hongqichen@lsec.cc.ac.cn)},
\ \ Hehu Xie\footnote{Corresponding author: LSEC, ICMSEC, Academy of Mathematics and Systems Science,
Chinese Academy of Sciences, Beijing 100190, P.R. China,  and School of Mathematical Sciences, University
of Chinese Academy of Sciences, Beijing, 100049, China (hhxie@lsec.cc.ac.cn)},
\ \ Meiling Yue\footnote{LSEC, ICMSEC, Academy of Mathematics and Systems Science,
Chinese Academy of Sciences, Beijing 100190, P.R. China, and School of Mathematical Sciences, University
of Chinese Academy of Sciences, Beijing, 100049, China (yuemeiling@lsec.cc.ac.cn)}\ \ \ and \
Chunguang You\footnote{LSEC, ICMSEC, Academy of Mathematics and Systems Science,
Chinese Academy of Sciences, Beijing 100190, P.R. China, and School of Mathematical Sciences, University
of Chinese Academy of Sciences, Beijing, 100049, China (youchg@lsec.cc.ac.cn)}
}
\date{}
\maketitle
\begin{abstract}
This paper is to give a new understanding and applications of the subspace projection method for selfadjoint
eigenvalue problems. A new error estimate in the energy norm, which is induced by the stiff matrix,
of the subspace projection method for eigenvalue problems is given. The relation between error estimates
in $L^2$-norm and energy norm is also deduced. Based on this relation, a new type of inverse power method
is designed for eigenvalue problems and the corresponding convergence analysis is also provided.
Then we present the analysis of the geometric and algebraic multigrid methods for eigenvalue problems based
on the convergence result of the new inverse power method.


\end{abstract}
\section{Introduction}
Large scale eigenvalue problems always occur in discipline of science and engineering such as material science,
quantum chemistry or physics, structure mechanics, biological system, data and information fields, etc.
With the increasing of the size and complexity of eigenvalue problems, the efficient solvers become
very important and there is a strong demand by engineers and scientists for efficient eigenvalue solvers.
There exist many numerical methods for solving large scale eigenvalue problems which are
based on power iteration, Krylov subspace iteration and so on\cite{BaiDemmelDongarraRuheVorst,Saad1,Saad2,Yousef}.
The basic idea is the subspace projection method especially for the eigenvalue problems which come
from the discretization of differential operators. From this point of view, the Krylov subspace
\cite{BaiDemmelDongarraRuheVorst,Saad1,Saad2,Yousef} and LOBPCG \cite{Knyazev} methods can be seen
as providing ways to build the subspace. The subspace projection understanding also provides the idea
to find new methods to design efficient eigenvalue solves which will be discussed in this paper.

It is well known that the algebraic and geometric multigrid methods are efficient solvers for
linear equations which come from the discretization of partial differential equations \cite{Bramble,BrambleZhang,BrennerScott,Hackbusch_Book,McCormick,RugeStuben,Shaidurov,Stuben,TrottenbergOosterleeSchuller,Xu1992Iterative}.
So far, the corresponding theory, applications and software have already been developed very well.
A natural topic is to consider the applications of multigrid methods to eigenvalue problems.
In the past three decades, there have appeared many applications of the multigrid method to eigenvalue problems.
The paper \cite{KnyazevNeymeyr} gives a very good review of different multigrid- based approaches for numerical
solutions of eigenvalue problems.
So far, the designing of the multigrid- based methods depends strongly on eigenvalue solvers and the multigrid is used as an inner solver.
Recently, we propose a multilevel and multigrid schemes \cite{HanLiXie,LinXie_2011,LinXie_2012,LinXie,Xie_JCP,Xie_IMA,Xie_Nonconforming}
for eigenvalue problems which is based on an new idea to construct the subspace. This schemes decompose a large scale eigenvalue
problem into standard linear equations plus small scale eigenvalue problems such that the choices of linear solvers and eigensolvers
are both free.  We also find papers \cite{KuzlVanek,Pultarova,VanekPultarov} also give the similar method for symmetric
positive eigenvalue problems which is only based on the inverse iteration or Rayleigh quotient method.
Actually, the method and results in this paper can also provide a reasonable analysis for their schemes.

The theoretical analysis reveals that the multigrid method depends on the ellipticity which is included in the matrices since
the ellipticity can leads to the famous duality argument (Aubin-Nitsche technique) \cite{BrennerScott,Ciarlet}.
As we know the duality argument deduces the relation between the error estimates in weak and strong
norms, which is the most important basis for the multigrid method designing.
Unfortunately, there is no this type of duality argument in the numerical theory for algebraic
eigenvalue problems. It is well known that there exists the error estimate in $L^2$-norm
for the subspace approximation method \cite{Saad1,Saad2,Yousef}. This $L^2$-norm error estimate
leads the error estimates for the Laczos and Arnoldi methods for eigenvalue problems \cite{Yousef}.
In order to understand and design multigrid methods in the sense of subspace method for eigenvalue problems,
we give the error estimate in the energy norm and the relation between error estimates
in $L^2$-norm and energy norm. Furthermore, this relation provides an idea to design and understand
the geometric and algebraic multigrid methods for eigenvalue problems.

An outline of the paper goes as follows. In Section 2, we introduce the subspace method
for solving eigenvalue problem. A new energy error estimate and the relation between
error estimates in $L^2$-norm and energy norm are given in Section 3.
In Section 4, we design and analyze a new type of inverse power method on a special subspace  for
eigenvalue problem. Based on the results in Section 4, we will describe the geometric and algebraic
multigrid methods for the eigenvalue problem solving in Sections 5 and 6, respectively.
Some concluding remarks are given in the last section.

\section{Subspace projection method}
For clearity, we simply introduce some basic knowledge of subspace projection method
and some definitions here. In this paper, we are mainly concerned with the following algebraic
eigenvalue problem: Find $u\in\mathbb{R}^n$ and $\lambda\in\mathbb{R}$  such that
\begin{equation}\label{Eigenvalue_Equation_0}
Au=\lambda u,
\end{equation}
where $A\in\mathbb{R}^{n\times n}$ is a symmetric positive definite matrix.

For our description and analysis, we introduce the following $L^2$ and energy inner products
\begin{eqnarray*}
(x,y)=x^Ty,\ \ \ \ \ (x,y)_A = x^TAy.
\end{eqnarray*}
Then the corresponding $L^2$-norm and energy norm can be defined as follows
\begin{eqnarray*}
\|x\|_2:=\sqrt{(x,x)},\ \  \quad \|x\|_A:=\sqrt{(x,x)_A}=\sqrt{(Ax,x)}.
\end{eqnarray*}

Let $\mathcal K$ be a $m$-dimensional subspace of $\mathbb{R}^n$.
An orthogonal projection technique for the eigenvalue problem onto the subspace $\mathcal K$ is to
seek an approximate eigenpair $(\widetilde\lambda, \widetilde u)\in \mathbb R\times \mathcal K$
to the problem (\ref{Eigenvalue_Equation_0}),
such that the following Galerkin condition is satisfied:
\begin{equation}\label{Gakerkin_condition}
A \widetilde u-\widetilde \lambda\widetilde u \perp_2 \mathcal K,
\end{equation}
or equivalently
\begin{equation}\label{Projection_method}
(A\widetilde u-\widetilde \lambda \widetilde u, v)=0, \quad \forall v\in \mathcal K.
\end{equation}
It is well known that there exists an orthonormal basis $\{v_1,\cdots,v_m\}$ for $\mathcal{K}$
and we denote the matrix with this basis being the column vectors by $V$
\begin{eqnarray*}
V:=[v_1,\cdots,v_m]\in \mathbb R^{n\times m}.
\end{eqnarray*}
Based on the orthonormal basis $V$ of the subspace $\mathcal K$, we can transform the
original eigenvalue problem (\ref{Eigenvalue_Equation_0}) to a small-scale eigenvalue
problem (always it is called Ritz problem).
Let $\widetilde u=Vy$ with $y\in\mathbb R^m$.  Then the problem (\ref{Projection_method}) becomes
\begin{equation}
(AVy-\widetilde \lambda Vy,v_j)=0, \quad {\rm for}\ j=1,2,\cdots,m.
\end{equation}
So we just need to solve the following small-scale eigenproblem:
Find $y\in\mathbb R^m$ and $\widetilde \lambda\in\mathbb R$ such that
\begin{equation}\label{Ritz_Problem}
A_my=\widetilde \lambda y,
\end{equation}
where $A_m=V^{T}AV$.

In order to translate the subspace projection method into the operator form,
we define some projection operators.
\begin{definition}($L^2$-Projector)
The $L^2$-projection operator $\mathcal R_\mathcal{K}:\mathbb{R}^n\mapsto \mathcal{K}$
is defined as follows
\begin{eqnarray}\label{L2_Projection}
(\mathcal R_\mathcal{K}x,y)=(x,y),\quad \forall x\in \mathbb R^n\ {\rm and}\  \forall y\in \mathcal{K}.
\end{eqnarray}
\end{definition}
In order to give the error estimate of the subspace projection method in the norm $\|\cdot\|_A$,
we also define the following projection operator by the inner product $(\cdot,\cdot)_A$.
\begin{definition}($A$-Projector)
The Galerkin projection operator $\mathcal{P}_{\mathcal{K}}:\mathbb{R}^n\mapsto\mathcal{K}$
is defined as follows
\begin{eqnarray}\label{Energy_Projection}
(\mathcal{P}_{\mathcal{K}}x,y)_A=(x,y)_A,\quad \forall x\in\mathbb R^n\ {\rm and}\  \forall y\in \mathcal{K}.
\end{eqnarray}
\end{definition}

Based on the projection operator $\mathcal R_\mathcal K$,
the Galerkin condition (\ref{Gakerkin_condition}) is equivalent to the following condition
\begin{equation*}
\mathcal{R}_{\mathcal{K}}(A\widetilde u-\widetilde \lambda \widetilde u)=0,
\quad {\rm where}\ \widetilde\lambda\in \mathbb{R}\ \ {\rm and}\ \ \widetilde{u}\in\mathcal{K}.
\end{equation*}
The above equality can be written as
\begin{equation*}
\mathcal{R}_{\mathcal{K}}A \widetilde u=\widetilde \lambda\widetilde  u,
\quad {\rm where}\ \widetilde \lambda\in \mathbb{R}\ \ {\rm and}\ \ \widetilde u\in\mathcal{K}.
\end{equation*}
This operator $\mathcal R_\mathcal{K}A$ can be viewed as $\mathcal R_{\mathcal{K}}A|_\mathcal{K}$
from $\mathcal{K}$ to $\mathcal{K}$.

Now, we state the following error estimate for the subspace
projection method which also motivates the analysis in this paper.
\begin{lemma}(\cite[Theorem 4.6]{Yousef})
Let $\gamma=\|\mathcal R_\mathcal KA(I-\mathcal R_\mathcal K)\|_2$, and consider any eigenvalue
$\lambda$ of $A$ with associated eigenvector $u$. Let $\widetilde\lambda$ be the approximate
eigenvalue closest to $\lambda$ and $\delta$ the distance between $\lambda$ and the set of
approximate eigenvalues other than $\widetilde\lambda$. Then there exists an approximate
eigenvector $\widetilde u$ associated with $\widetilde\lambda$ such that
\begin{eqnarray}
\sin[\theta(u,\widetilde u)]&\leq& \sqrt{1+\frac{\gamma^2}{\delta^2}}\sin[\theta(u,\mathcal K)].
\end{eqnarray}
\end{lemma}
Now, we also state some properties about the energy projection operator $\mathcal P_\mathcal K$.
\begin{proposition}\label{A_orthor_proposition}
It is obvious that the following properties hold
\begin{eqnarray}
\mathcal{P}_{\mathcal{K}}^2=\mathcal{P}_{\mathcal{K}}
\ \ \ \ {\rm and}\ \ \ \
\mathcal{P}_{\mathcal{K}}(I-\mathcal{P}_{\mathcal{K}})=0.
\end{eqnarray}
\end{proposition}
\section{Energy error estimate}
In this section, we will give some new error estimates of the subspace method
for the eigenvalue problem. The new thing in this paper is that we establish error
estimates of the eigenfunction approximation in energy norm rather than $L^2$-norm.
Furthermore, the relation between the energy norm and $L^2$-norm is also derived here.

We can order the eigenvalues of the matrix $A$ as the following
increasing sequence
\begin{equation*}
\lambda_1\leq\lambda_2\leq\cdots \leq \lambda_n,
\end{equation*}
and the corresponding eigenvectors
\begin{equation*}
u_1,u_2,\cdots,u_n,
\end{equation*}
where $(u_i,u_j)_A=\delta_{ij}$ for $i,j=1,...,n$ and $\delta_{ij}$ is Kronecker notation.
Similarly, the eigenpairs of $A_m$ can be ordered as follows
\begin{equation}\label{Lambda_Tilde}
\widetilde{\lambda}_1\leq \widetilde{\lambda}_2\leq\cdots \leq \widetilde{\lambda}_m,
\end{equation}
and
\begin{equation*}
\widetilde{u}_1,\widetilde{u}_2,\cdots,\widetilde{u}_m,
\end{equation*}
where $(\widetilde{u}_i,\widetilde{u}_j)_A=\delta_{ij}$ for $i,j=1,...,m$.


From the min-max principle for eigenvalue problems, the following
upper bound property holds.
\begin{proposition}(\cite[Corollary 4.1]{Yousef})
The following inequality holds
\begin{equation}\label{Upper_Bound}
\lambda_i\leq  \widetilde{\lambda}_i, \quad i=1,2,\cdots,m.
\end{equation}
\end{proposition}

In this paper, we solve the eigenvalue problem on the subspace $\mathcal K$ to
obtain the eigenpair approximations
$(\widetilde\lambda_1,\widetilde u_1), \cdots, (\widetilde\lambda_k,\widetilde u_k)$
for $k$ exact eigenpairs of eigenvalue problem (\ref{Eigenvalue_Equation_0}).
In order to deduce the error estimate in the energy norm, we write the eigenvalue
problem (\ref{Eigenvalue_Equation_0}) as the following version: Find $u\in \mathbb R^n$ and $\mu\in\mathbb R$
such that
\begin{equation}\label{Eigenvalue_Equation}
A^{-1}u=\mu u
\end{equation}
or the following variational form
\begin{equation}\label{Weak_Eigenvalue_Equation}
(A^{-1}u,v)_A=\mu (u,v)_A,\ \ \ \forall v\in \mathbb R^n.
\end{equation}
It is easy to know that $\mu=1/\lambda$.
For simplicity of notation, we denote $A^{-1}$ by $T$ in this paper.
We know that $T$ has the same eigenvector as $A$ with the eigenvalue $\mu$.
Now the Galerkin equation for the approximation on the subspace $\mathcal K$
is defined as follows
\begin{eqnarray}\label{Galerkin_Equation}
(\mathcal P_\mathcal KT\widetilde u,v)_A=\widetilde \mu (\widetilde u,v)_A,
\ \ \ \forall v\in\mathcal K.
\end{eqnarray}
The equation (\ref{Galerkin_Equation}) can also be written as the following operator form
\begin{eqnarray}\label{Subspace_Operator_Eigen}
\mathcal P_\mathcal K T\widetilde u=\widetilde \mu \widetilde u.
\end{eqnarray}
Obviously, the eigenvalue problem (\ref{Subspace_Operator_Eigen}) has the eigenvalues
$\widetilde \mu_1\geq\cdots\geq \widetilde\mu_m$ and the corresponding
eigenvectors $\widetilde u_1,\cdots,\widetilde u_m$.
We can also know that $\widetilde\mu_i=1/\widetilde\lambda_i$.

Before stating the error estimates of the subspace projection method, we introduce a lemma which
comes from \cite{StrangFix}. For completeness, a proof is also stated here. 
\begin{lemma}(\cite[Lemma 6.4]{StrangFix})\label{Strang_Lemma}
For any exact eigenpair $(\lambda,u)$ of (\ref{Eigenvalue_Equation_0}), the following equality holds
\begin{eqnarray}\label{Strang_Equality}
(\widetilde\lambda_j-\lambda)(\mathcal P_\mathcal Ku,\widetilde u_j)
=\lambda(u-\mathcal P_\mathcal Ku,\widetilde u_j),\ \ \
j = 1, ...,m.
\end{eqnarray}
\end{lemma}
\begin{proof}
Since $-\lambda (\mathcal P_\mathcal Ku,\widetilde u_j)$ appears on both sides, we only need to prove
that
\begin{eqnarray*}
\widetilde\lambda_j(\mathcal P_\mathcal Ku,\widetilde u_j)=\lambda (u,\widetilde u_j).
\end{eqnarray*}
From (\ref{Eigenvalue_Equation_0}), (\ref{Projection_method}) and (\ref{Energy_Projection}),
the following equalities hold
\begin{eqnarray*}
\widetilde\lambda_j(\mathcal P_\mathcal Ku,\widetilde u_j) = (\mathcal P_\mathcal Ku,\widetilde u_j)_A
=(u,\widetilde u_j)_A = \lambda(u,\widetilde u_j).
\end{eqnarray*}
Then the proof is complete.
\end{proof}
\begin{theorem}\label{Error_Estimate_Theorem}
Let  $(\lambda,u)$ denote an exact eigenpair of the eigenvalue problem (\ref{Eigenvalue_Equation_0}).
Assume the eigenpair approximation $(\widetilde\lambda_i,\widetilde u_i)$ has the property that
$\widetilde\mu_i=1/\widetilde\lambda_i$ is closest to $\mu=1/\lambda$.
The corresponding spectral projection $E_m^{(i)}: \mathbb R^n\mapsto {\rm span}\{\widetilde u_i\}$ is defined as follows
\begin{eqnarray*}
(E_m^{(i)}w,\widetilde u_i)_A = (w,\widetilde u_i)_A,\ \ \ \ \forall w\in \mathbb R^n.
\end{eqnarray*}
Then the following error estimate holds
\begin{eqnarray}\label{Energy_Error_Estimate}
\|u-E_m^{(i)}u\|_A&\leq& \sqrt{1+\frac{\widetilde\mu_1}{\delta_i^2}\eta_\mathcal K^2}
\|(I-\mathcal P_\mathcal K)u\|_A,
\end{eqnarray}
where $\eta_\mathcal K$ and $\delta_i$ are defined as follows
\begin{eqnarray}
\eta_\mathcal K&:=&\sup_{\|g\|_2=1}\|(I-\mathcal P_\mathcal K)Tg\|_A
=\sup_{\|g\|_2=1}\|(I-\mathcal P_\mathcal K)A^{-1}g\|_A\label{Definition_Gamma_},\\
\delta_i &:=& \min_{j\neq i}|\widetilde \mu_j-\mu|=\min_{j\neq i}
\Big|\frac{1}{\widetilde\lambda_j}-\frac{1}{\lambda}\Big|.\label{Definition_Delta}
\end{eqnarray}
Furthermore, the eigenvector approximation $\widetilde u_i$ has the following
error estimate in the $L^2$-norm
\begin{eqnarray}\label{L2_Error_Estimate}
\|u-E_m^{(i)}u\|_2 &\leq&\eta_{\mathcal K,i}\|u-E_m^{(i)}u\|_A,
\end{eqnarray}
where $\eta_{\mathcal K,i}$ is defined as follows
\begin{eqnarray}\label{Definition_Gamma_2}
\eta_{\mathcal K,i} = \Big(1+\frac{\widetilde\mu_1}{\delta_i}\Big)\eta_\mathcal K.
\end{eqnarray}
\end{theorem}
\begin{proof}
Similarly to the duality argument in the finite element method, the following inequality holds
\begin{eqnarray}\label{L2_Energy_Estiate}
&&\|(I-\mathcal P_\mathcal K)u\|_2=\sup_{\|g\|_2=1}((I-\mathcal P_{\mathcal K})u,g)
=\sup_{\|g\|_2=1}((I-\mathcal P_{\mathcal K})u,Tg)_A \nonumber\\
&&=\sup_{\|g\|_2=1}((I-\mathcal P_{\mathcal K})u,(I-\mathcal P_\mathcal K)Tg)_A
\leq \eta_\mathcal K\|(I-\mathcal P_\mathcal K)u\|_A.
\end{eqnarray}

Since $(I-E_m^{(i)})\mathcal P_\mathcal Ku\in\mathcal K$ and
$(I-E_m^{(i)})\mathcal P_\mathcal Ku\perp_A \widetilde u_i$,
the following orthogonal expansion holds
\begin{eqnarray}\label{Orthogonal_Decomposition}
(I-E_m^{(i)})\mathcal P_\mathcal Ku=\sum_{j\neq i}\alpha_j\widetilde u_j,
\end{eqnarray}
where $\alpha_j=(\mathcal P_\mathcal Ku,\widetilde u_j)_A$. From Lemma \ref{Strang_Lemma}, we have
\begin{eqnarray}\label{Alpha_Estimate}
\alpha_j=(\mathcal P_\mathcal Ku,\widetilde u_j)_A = \widetilde\lambda_j\big(\mathcal P_\mathcal Ku,\widetilde u_j\big)
=\frac{\widetilde\lambda_j\lambda}{\widetilde\lambda_j-\lambda}\big(u-\mathcal P_\mathcal Ku,\widetilde u_j\big)
=\frac{1}{\mu-\widetilde\mu_j}\big(u-\mathcal P_\mathcal Ku,\widetilde u_j\big).
\end{eqnarray}
From the property of the eigenvectors $\widetilde u_1, ..., \widetilde u_m$, the following equalities hold
\begin{eqnarray*}
1 = (\widetilde u_j,\widetilde u_j)_A = \widetilde\lambda_j(\widetilde u_j,\widetilde u_j)=\widetilde\lambda_j\|\widetilde u_j\|_2^2,
\end{eqnarray*}
which leads to the following property
\begin{eqnarray}\label{Equality_u_j}
\|\widetilde u_j\|_2^2=\frac{1}{\widetilde\lambda_j}=\widetilde\mu_j.
\end{eqnarray}
From (\ref{Projection_method}) and the definitions of the eigenvectors $\widetilde u_1, ..., \widetilde u_m$,
we have the following equalities
\begin{eqnarray}\label{Orthonormal_Basis}
(\widetilde u_j,\widetilde u_k)_A=\delta_{jk},\ \ \ \ \ \Big(\frac{\widetilde u_j}{\|\widetilde u_j\|_2},\frac{\widetilde u_k}{\|\widetilde u_k\|_2}\Big)=\delta_{jk},\ \ \ 1\leq j,k\leq m.
\end{eqnarray}
Then from (\ref{Orthogonal_Decomposition}), (\ref{Alpha_Estimate}), (\ref{Equality_u_j}) and (\ref{Orthonormal_Basis}),
the following estimates hold
\begin{eqnarray}\label{Equality_5}
&&\|(I-E_m^{(i)})\mathcal P_\mathcal Ku\|_A^2=\Big\|\sum_{j\neq i}\alpha_j\widetilde u_j\Big\|_A^2
=\sum_{j\neq i}\alpha_j^2=\sum_{j\neq i}\Big(\frac{1}{\mu-\widetilde\mu_j}\Big)^2\big(u-\mathcal P_\mathcal Ku,\widetilde u_j\big)^2\nonumber\\
&&\leq\frac{1}{\delta_i^2}\sum_{j\neq i}\|\widetilde u_j\|_2^2\Big(u-\mathcal P_\mathcal Ku,\frac{\widetilde u_j}{\|\widetilde u_j\|_2}\Big)^2
=\frac{1}{\delta_i^2}\sum_{j\neq i}\widetilde\mu_j\Big(u-\mathcal P_\mathcal Ku,\frac{\widetilde u_j}{\|\widetilde u_j\|_2}\Big)^2\nonumber\\
&&\leq \frac{\widetilde\mu_1}{\delta_i^2}\|u-\mathcal P_\mathcal Ku\|_2^2.
\end{eqnarray}
From (\ref{L2_Energy_Estiate}), (\ref{Equality_5}) and the orthogonal property
$u-\mathcal P_\mathcal Ku\perp_A(I-E_m^{(i)})\mathcal P_\mathcal Ku$,
we have the following error estimate
\begin{eqnarray}
&&\|u-E_m^{(i)}u\|_A^2=\|u-\mathcal P_\mathcal Ku\|_A^2
+\|(I-E_m^{(i)})\mathcal P_\mathcal Ku\|_A^2\nonumber\\
&&\leq\|(I-\mathcal P_\mathcal K)u\|_A^2
+\frac{\widetilde\mu_1}{\delta_i^2}\|u-\mathcal P_\mathcal Ku\|_2^2
\leq \Big(1+\frac{\widetilde\mu_1}{\delta_i^2}\eta_\mathcal K^2\Big)\|(I-\mathcal P_\mathcal K)u\|_A^2.
\end{eqnarray}
This is the desired result (\ref{Energy_Error_Estimate}).

Similarly, from (\ref{Orthogonal_Decomposition}), (\ref{Alpha_Estimate}), (\ref{Equality_u_j}) and (\ref{Orthonormal_Basis}),
the following estimates hold
\begin{eqnarray}\label{Equality_6}
&&\|(I-E_m^{(i)})\mathcal P_\mathcal Ku\|_2^2 = \Big\|\sum_{j\neq i}\alpha_j\widetilde u_j\Big\|_2^2
=\sum_{j\neq i}\alpha_j^2\|\widetilde u_j\|_2^2 \nonumber\\
&&= \sum_{j\neq i}\Big(\frac{1}{\mu-\widetilde \mu_j}\Big)^2\big(u-\mathcal P_\mathcal Ku,\widetilde u_j\big)^2\|\widetilde u_j\|_2^2
\leq\Big(\frac{1}{\delta_i}\Big)^2\sum_{j\neq i}\Big(u-\mathcal P_\mathcal Ku,\frac{\widetilde u_j}{\|\widetilde u_j\|_2}\Big)^2\|\widetilde u_j\|_2^4\nonumber\\
&&=\Big(\frac{1}{\delta_i}\Big)^2\sum_{j\neq i}\widetilde\mu_j^2\Big(u-\mathcal P_\mathcal Ku,\frac{\widetilde u_j}{\|\widetilde u_j\|_2}\Big)^2
\leq \Big(\frac{\widetilde\mu_1}{\delta_i}\Big)^2\|u-\mathcal P_\mathcal Ku\|_2^2.
\end{eqnarray}
Combining (\ref{L2_Energy_Estiate}) and (\ref{Equality_6}) leads to the following inequalities
\begin{eqnarray}\label{Equality_8}
\|(I-E_m^{(i)})\mathcal P_\mathcal Ku\|_2 \leq \frac{\widetilde\mu_1}{\delta_i}\|u-\mathcal P_\mathcal Ku\|_2
\leq \frac{\widetilde\mu_1}{\delta_i}\eta_{\mathcal K}\|(I-\mathcal P_\mathcal K)u\|_A.
\end{eqnarray}
From (\ref{L2_Energy_Estiate}), (\ref{Equality_8}) and the triangle inequality,
we have the following error estimate for the eigenvector approximation in the $L^2$-norm
\begin{eqnarray}\label{Inequality_6}
&&\|u-E_m^{(i)}u\|_2\leq \|u-\mathcal P_\mathcal Ku\|_2
+ \|(I-E_m^{(i)})\mathcal P_\mathcal Ku\|_2\nonumber\\
&\leq&\|u-\mathcal P_\mathcal Ku\|_2
+ \frac{\widetilde\mu_1}{\delta_i}\eta_\mathcal K\|(I-\mathcal P_\mathcal K)u\|_A\nonumber\\
&\leq&\Big(1+\frac{\widetilde\mu_1}{\delta_i}\Big)\eta_\mathcal K\|(I-\mathcal P_\mathcal K)u\|_A
\leq \eta_{\mathcal K,i}\|u-E_m^{(i)}u\|_A.
\end{eqnarray}
This is the second desired result (\ref{L2_Error_Estimate}) and the proof is complete.
\end{proof}

In the following analysis, we state the error estimates for multi eigenvalue approximations. For simplicity of notation,
we consider the special case that the first $k$ eigenvalues $\lambda_1\leq\cdots\leq\lambda_k$ are closest to
the eigenvalue approximations $\widetilde\lambda_1\leq\cdots\leq\widetilde\lambda_k$. Then the corresponding
eigenvector approximations $\widetilde u_1,\cdots,\widetilde u_k$ have the error estimates stated in the next theorem.
\begin{theorem}\label{Error_Estimate_Theorem_k}
We define spectral projection
$E_{m,k}:\mathbb R^n\mapsto{\rm span}\{\widetilde u_1, ..., \widetilde u_k\}$
corresponding to the first $k$ eigenvector approximations
$\widetilde u_1,\cdots,\widetilde u_k$ as follows
\begin{eqnarray*}
(E_{m,k}w,\widetilde u_i)_A = (w,\widetilde u_i)_A,
\ \ \ \ \forall w\in \mathbb R^n\ \ {\rm and}\ \ i=1,\cdots,k.
\end{eqnarray*}
Then the associated exact eigenvectors $u_1$, $...$, $u_k$ of
problem (\ref{Eigenvalue_Equation_0}) have the following error estimate
\begin{eqnarray}\label{Energy_Error_Estimate_k}
\|u_i-E_{m,k}u_i\|_A&\leq& \sqrt{1+\frac{\widetilde\mu_{k+1}}{\delta_{k,i}^2}\eta_\mathcal K^2}
\|(I-\mathcal P_\mathcal K)u_i\|_A,\ \ \ 1\leq i\leq k,
\end{eqnarray}
where $\delta_{k,i}$ is defined as follows
\begin{eqnarray}\label{Definition_Gamma_Delta_k}
\delta_{k,i} := \min_{k<j\leq m}|\widetilde \mu_j-\mu_i|=\min_{k<j\leq m}
\Big|\frac{1}{\widetilde\lambda_j}-\frac{1}{\lambda_i}\Big|.
\end{eqnarray}
Furthermore, these $k$ exact eigenvectors have the following error estimate in the $L^2$-norm
\begin{eqnarray}\label{L2_Error_Estimate_k}
\|u_i-E_{m,k}u_i\|_2 &\leq&\eta_{\mathcal K,k,i}\|u_i-E_{m,k}u_i\|_A,\ \ \ 1\leq i\leq k,
\end{eqnarray}
where $\eta_{\mathcal K,k,i}$ is defined as follows
\begin{eqnarray}
\eta_{\mathcal K,k,i}= \Big(1+\frac{\widetilde\mu_{k+1}}{\delta_{k,i}}\Big)\eta_\mathcal K.
\end{eqnarray}
\end{theorem}
\begin{proof}
Since $(I-E_{m,k})\mathcal P_\mathcal Ku_i\in\mathcal K$ and
$(I-E_{m,k})\mathcal P_\mathcal Ku_i\in {\rm span}\{\widetilde u_{k+1},...,\widetilde u_m\}$,
the following orthogonal expansion holds
\begin{eqnarray}\label{Orthogonal_Decomposition_k}
(I-E_{m,k})\mathcal P_\mathcal Ku_i=\sum_{j=k+1}^m\alpha_j\widetilde u_j.
\end{eqnarray}
Then from (\ref{Alpha_Estimate}), (\ref{Equality_u_j}), (\ref{Orthonormal_Basis}) and (\ref{Orthogonal_Decomposition_k}),
we have the following estimates
\begin{eqnarray}\label{Equality_4_i}
&&\|(I-E_{m,k})\mathcal P_\mathcal Ku_i\|_A^2 = \Big\|\sum_{j=k+1}^m\alpha_j\widetilde u_j\Big\|_A^2
= \sum_{j=k+1}^m\alpha_j^2\nonumber\\
&&=\sum_{j=k+1}^m\Big(\frac{1}{\mu_i-\widetilde\mu_j}\Big)^2\big(u_i-\mathcal P_\mathcal Ku_i,\widetilde u_j\big)^2
\leq\frac{1}{\delta_{k,i}^2}\sum_{j=k+1}^m\|\widetilde u_j\|_2^2\Big(u_i-\mathcal P_\mathcal Ku_i,\frac{\widetilde u_j}{\|\widetilde u_j\|_2}\Big)^2\nonumber\\
&&=\frac{1}{\delta_{k,i}^2}\sum_{j=k+1}^m\widetilde\mu_j\Big(u_i-\mathcal P_\mathcal Ku_i,\frac{\widetilde u_j}{\|\widetilde u_j\|_2}\Big)^2
\leq \frac{\widetilde\mu_{k+1}}{\delta_{k,i}^2}\|u_i-\mathcal P_\mathcal Ku_i\|_2^2.
\end{eqnarray}
Similarly combining (\ref{L2_Energy_Estiate}) and (\ref{Equality_4_i}) leads to the following inequality
\begin{eqnarray}\label{Equality_5_k}
\|(I-E_{m,k})\mathcal P_\mathcal Ku_i\|_A^2
\leq\frac{\widetilde\mu_{k+1}}{\delta_{k,i}^2}\eta_\mathcal K^2\|(I-\mathcal P_\mathcal K)u_i\|_A^2.
\end{eqnarray}
From (\ref{Equality_5_k}) and the orthogonal property
$u_i-\mathcal P_\mathcal Ku_i\perp_A (I-E_{m,k})\mathcal P_\mathcal Ku_i$,
we have the following error estimate
\begin{eqnarray*}
\|u_i-E_{m,k}u_i\|_A^2&=&\|u_i-\mathcal P_\mathcal Ku_i\|_A^2
+\|(I-E_{m,k})\mathcal P_\mathcal Ku_i\|_A^2\nonumber\\
&\leq&\Big(1+\frac{\widetilde\mu_{k+1}}{\delta_{k,i}^2}\eta_\mathcal K^2\Big)
\|(I-\mathcal P_\mathcal K)u_i\|_A^2.
\end{eqnarray*}
This is the desired result (\ref{Energy_Error_Estimate_k}).

Similarly, from (\ref{Alpha_Estimate}), (\ref{Equality_u_j}), (\ref{Orthonormal_Basis}) and (\ref{Orthogonal_Decomposition_k}),
we have the following estimates
\begin{eqnarray*}
&&\|(I-E_{m,k})\mathcal P_\mathcal Ku_i\|_2^2 = \|\sum_{j=k+1}^m\alpha_j\widetilde u_j\|_2^2
= \sum_{j=k+1}^m\alpha_j^2\|\widetilde u_j\|_2^2\nonumber\\
&&=\sum_{j=k+1}^m\Big(\frac{1}{\mu_i-\widetilde\mu_j}\Big)^2\big(u_i-\mathcal P_\mathcal Ku_i,\widetilde u_j\big)^2\|\widetilde u_j\|_2^2
\leq\frac{1}{\delta_{k,i}^2}\sum_{j=k+1}^m\|\widetilde u_j\|_2^4\Big(u_i-\mathcal P_\mathcal Ku_i,\frac{\widetilde u_j}{\|\widetilde u_j\|_2}\Big)^2\nonumber\\
&&=\frac{1}{\delta_{k,i}^2}\sum_{j=k+1}^m\widetilde\mu_j^2\Big(u_i-\mathcal P_\mathcal Ku_i,\frac{\widetilde u_j}{\|\widetilde u_j\|_2}\Big)^2
\leq \frac{\widetilde\mu_{k+1}^2}{\delta_{k,i}^2}\|u_i-\mathcal P_\mathcal Ku_i\|_2^2,
\end{eqnarray*}
which leads to the inequality
\begin{eqnarray}\label{Equality_8_k}
\|(I-E_{m,k})\mathcal P_\mathcal Ku_i\|_2 \leq \frac{\widetilde\mu_{k+1}}{\delta_{k,i}}\|u_i-\mathcal P_\mathcal Ku_i\|_2.
\end{eqnarray}
From (\ref{L2_Energy_Estiate}), (\ref{Equality_8_k}) and the triangle inequality, we have
the following error estimate for the eigenvector approximation in the $L^2$-norm
\begin{eqnarray*}\label{Inequality_11}
&&\|u_i-E_{m,k}u_i\|_2\leq \|u_i-\mathcal P_\mathcal Ku_i\|_2 + \|(I-E_{m,k})
\mathcal P_\mathcal Ku_i\|_2\nonumber\\
&\leq&\Big(1+\frac{\widetilde\mu_{k+1}}{\delta_{k,i}}\Big)
\|(I-\mathcal P_\mathcal K)u_i\|_2\leq \Big(1+\frac{\widetilde\mu_{k+1}}{\delta_{k,i}}\Big)\eta_\mathcal K\|(I-\mathcal P_\mathcal K)u_i\|_A\nonumber\\
&\leq& \Big(1+\frac{\widetilde\mu_{k+1}}{\delta_{k,i}}\Big)\eta_\mathcal K\|u_i-E_{m,k}u_i\|_A.
\end{eqnarray*}
This is the second desired result (\ref{L2_Error_Estimate_k}) and the proof is complete.
\end{proof}
\section{Inverse power method on a subspace}

As an application of the error estimates stated in the previous section, we give an
algebraic error estimate for the inverse power method on a special subspace which is
constructed by enriching the current eigenspace approximation with a space $\mathcal K$.
For more details about this special space, please refer to \cite{LinXie,Xie_JCP}.

For some given eigenvector approximations $u_1^{(\ell)}, ..., u_k^{(\ell)}$ which are
approximations for the first $k$ eigenvectors $u_1,...,u_k$, we do the following
inverse power iteration on a subspace:
\begin{algorithm}\label{Algorithm_1} {\bf Inverse power method on a subspace}\\
For given eigenvector approximations $u_1^{(\ell)}, ..., u_k^{(\ell)}$, do following two steps
\begin{enumerate}
\item Define the subspace $\mathcal K_k^{(\ell+1)}
:=\mathcal K+{\rm span}\{u_1^{(\ell)},...,u_k^{(\ell)}\}$
and solve the following eigenvalue problem:
Find $\widetilde u_i^{(\ell+1)}\in \mathcal K_k^{(\ell+1)}$ and $\lambda_i^{(\ell+1)}\in\mathbb R$
such that $\|\widetilde u_i^{(\ell+1)}\|_A=1$ and
\begin{eqnarray}\label{Subspace_Eigen_Problem}
(A \widetilde u_i^{(\ell+1)}, v) = \lambda_i^{(\ell+1)}(\widetilde u_i^{(\ell+1)},v),
\ \ \ \ \forall v\in \mathcal K_k^{(\ell+1)}.
\end{eqnarray}
Solve this eigenvalue problem to obtain the new first $k$ eigenvector approximations
$\widetilde u_1^{(\ell+1)},...,\widetilde u_k^{(\ell+1)}$.
\item Solve the following $k$ linear equations:
\begin{eqnarray}\label{Linear_Equation}
A u_i^{(\ell+1)}=\lambda_i^{(\ell+1)}\widetilde u_i^{(\ell+1)},\ \ \ \ i=1,...,k.
\end{eqnarray}
\end{enumerate}
We obtain the new eigenvector approximations $u_1^{(\ell+1)},..., u_k^{(\ell+1)}$ as the output.
\end{algorithm}
We define spectral projection $E_{m,k}^{(\ell)}:\mathbb R^n\mapsto {\rm span}\{u_1^{(\ell)}, ..., u_k^{(\ell)}\}$
corresponding to the eigenvector approximations $u_1^{(\ell)},\cdots, u_k^{(\ell)}$ as follows
\begin{eqnarray*}
(E_{m,k}^{(\ell)}w,u_i^{(\ell)})_A = (w, u_i^{(\ell)})_A,
\ \ \ \ \forall w\in \mathbb R^n\ \ {\rm and}\ \ i=1,\cdots,k.
\end{eqnarray*}
Then the spectral projections  $E_{m,k}^{(\ell+1)}$, $\tilde E_{m,k}^{(\ell+1)}$ and $E_k$ can also be defined
corresponding to the spaces ${\rm span}\{u_1^{(\ell+1)},..., u_k^{(\ell+1)}\}$,
${\rm span}\{\widetilde u_1^{(\ell+1)},..., \widetilde u_k^{(\ell+1)}\}$ and ${\rm span}\{u_1,..., u_k\}$, respectively.
Based on Theorems \ref{Error_Estimate_Theorem} and \ref{Error_Estimate_Theorem_k}, we give the following error
estimate for Algorithm \ref{Algorithm_1}.
\begin{theorem}\label{Theorem_Error_Estimate_Inverse}
There exist exact eigenvectors $u_1,..., u_k$ such that
the resultant eigenvector approximations $u_1^{(\ell+1)},..., u_k^{(\ell+1)}$ have the following error estimate
\begin{eqnarray}\label{Error_Estimate_Inverse}
\Big(\sum_{i=1}^k\|u_i-E_{m,k}^{(\ell+1)}u_i\|_A^2\Big)^{1/2}
\leq \theta_{\mathcal K_k^{(\ell+1)}}\sqrt{\frac{\lambda_k}{\lambda_{k+1}}}
\Big(\sqrt{\lambda_k^{(\ell+1)}}\eta_{\mathcal K_k^{(\ell+1)},k,k}\Big)
\Big(\sum_{i=1}^k\|u_i-E_{m,k}^{(\ell)}u_i\|_A^2\Big)^{1/2},
\end{eqnarray}
where $\theta_{\mathcal K_k^{(\ell+1)}}$ and $\eta_{\mathcal K_k^{(\ell+1)},k,k}$ are defined as follows
\begin{eqnarray}
\theta_{\mathcal K_k^{(\ell+1)}}&:=&
\sqrt{1+\frac{\mu_{k+1}^{(\ell+1)}\eta_{\mathcal K_k^{(\ell+1)}}^2}{\big(\delta_{k,k}^{(\ell+1)}\big)^2}},\label{Definition_Inverse_2}\\
\eta_{\mathcal K_k^{(\ell+1)},k,i} &:=& \Big(1+\frac{\mu_{k+1}^{(\ell+1)}}{\delta_{k,i}^{(\ell+1)}}\Big)
\eta_{\mathcal K_k^{(\ell+1)}}, \ \ \ i=1, ..., k,\label{Definition_Inverse_1}
\end{eqnarray}
with
\begin{eqnarray}
\delta_{k,i}^{(\ell+1)}:=\min_{k<j\leq m}|\mu_j^{(\ell+1)}-\mu_i|,\ i=1, ..., k \ {\rm and} \ \
\mu_i^{(\ell+1)}:=\frac{1}{\lambda_i^{(\ell+1)}}, \  i=1, ..., m.
\end{eqnarray}
\end{theorem}
\begin{proof}
From Theorem \ref{Error_Estimate_Theorem_k}, there exist exact eigenvectors $u_1, ..., u_k$ such that
the following error estimates for the eigenvector approximations
$\widetilde u_1^{(\ell+1)},...,\widetilde u_k^{(\ell+1)}$ hold for $i=1, ..., k$
\begin{eqnarray}
\|u_i-\tilde E_{m,k}^{(\ell+1)}u_i\|_A
&\leq& \sqrt{1+\frac{\mu_{k+1}^{(\ell+1)}\eta_{\mathcal K_k^{(\ell+1)}}^2}{\big(\delta_{k,i}^{(\ell+1)}\big)^2}}
\|(I-\mathcal P_{\mathcal K_k^{(\ell+1)}})u_i\|_A\nonumber\\
&\leq&  \sqrt{1+\frac{\mu_{k+1}^{(\ell+1)}\eta_{\mathcal K_k^{(\ell+1)}}^2}{\big(\delta_{k,i}^{(\ell+1)}\big)^2}}
\|u_i-E_{m,k}^{(\ell)}u_i\|_A,
\end{eqnarray}
and
\begin{eqnarray}\label{Inequality_13}
&&\|u_i-\tilde E_{m,k}^{(\ell+1)}u_i\|_2\leq \eta_{\mathcal K_k^{(\ell+1)},k,i}
\|u_i-\tilde E_{m,k}^{(\ell+1)}u_i\|_A\nonumber\\
&&\ \ \ \ \leq\eta_{\mathcal K_k^{(\ell+1)},k,i}
\sqrt{1+\frac{\mu_{k+1}^{(\ell+1)}\eta_{\mathcal K_k^{(\ell+1)}}^2}{\big(\delta_{k,i}^{(\ell+1)}\big)^2}}
\|u_i-E_{m,k}^{(\ell)}u_i\|_A.
\end{eqnarray}
Let $\alpha_i=1/\|u_i^{(\ell+1)}\|_A$.
From (\ref{Subspace_Eigen_Problem}) and (\ref{Linear_Equation}), we have following inequalities
\begin{eqnarray*}
1= \lambda_i^{(\ell+1)}(\widetilde u_i^{(\ell+1)},\widetilde u_i^{(\ell+1)})
= (A u_i^{(\ell+1)},\widetilde u_i^{(\ell+1)})\leq \|u_i^{(\ell+1)}\|_A\|\widetilde u_i^{(\ell+1)}\|_A=\frac{1}{\alpha_i}.
\end{eqnarray*}
Then each $\alpha_i$ has the following estimate
\begin{eqnarray}\label{Alpha_i_Estimate}
\alpha_i\leq 1,\ \ \ \ {\rm for}\ i=1,...,k.
\end{eqnarray}
For the analysis, we define the $L^2$-projections $\pi_k$ and $\tilde\pi_{m,k}^{(\ell+1)}$
corresponding to the spaces ${\rm span}\{u_1$,..., $u_k\}$ and
${\rm span}\{\widetilde u_1^{(\ell+1)},..., \widetilde u_k^{(\ell+1)}\}$, respectively.
Then since $\|u_i\|_A=\|\alpha_iu_i^{(\ell+1)}\|_A=1$ and
$\|\sqrt{\lambda_i}u_i\|_2=\|\sqrt{\lambda_i^{(\ell+1)}}\widetilde u_i^{(\ell+1)}\|_2=1$, there exist following equalities
\begin{eqnarray}\label{Equality_9}
\sum_{i=1}^k\|u_i-E_{m,k}^{(\ell+1)}u_i\|_A^2
= \sum_{i=1}^k\|\alpha_i u_i^{(\ell+1)}-E_k(\alpha_i u_i^{(\ell+1)})\|_A^2
\end{eqnarray}
and
\begin{eqnarray}\label{Equality_10_2}
\sum_{i=1}^k\Big\|\sqrt{\lambda_i^{(\ell+1)}}\widetilde u_i^{(\ell+1)}
-\pi_k \big(\sqrt{\lambda_i^{(\ell+1)}}\widetilde u_i^{(\ell+1)}\big)\Big\|_2^2
=\sum_{i=1}^k\|\sqrt{\lambda_i}u_i-\tilde\pi_{m,k}^{(\ell+1)}\big(\sqrt{\lambda_i}u_i\big)\|_2^2.
\end{eqnarray}
From the definition of the spectral projection $E_k$, it is easy to know the following property holds
\begin{eqnarray}\label{Inequality_Spectral_Projection}
\frac{\|u_i^{(\ell+1)}-E_ku_i^{(\ell+1)}\|_A^2}{\|u_i^{(\ell+1)}-E_ku_i^{(\ell+1)}\|_2^2}
\geq \lambda_{k+1}.
\end{eqnarray}
We define $\mathbf U=[u_1,...,u_k]\in\mathbb R^{n\times k}$,
$\widetilde{\mathbf U}^{(\ell+1)}=[\widetilde u_1^{(\ell+1)},...,
\widetilde u_k^{(\ell+1)}]\in \mathbb R^{n\times k}$
and $\mathbf U^{(\ell+1)}$ $=[u_1^{(\ell+1)}$, ..., $u_k^{(\ell+1)}]\in \mathbb R^{n\times k}$.
It is easy to know that there exists a nonsingular matrix $Q\in \mathbb R^{k\times k}$ such that
\begin{eqnarray}\label{Equality_10}
\pi_k\widetilde{\mathbf U}^{(\ell+1)}=[\pi_k\widetilde u_1^{(\ell+1)},...,
\pi_k\widetilde u_k^{(\ell+1)}]=\mathbf U Q.
\end{eqnarray}
From the definition of spectral projection  $E_k$, the following equation holds
\begin{eqnarray}\label{Equality_11}
(\mathbf U^{(\ell+1)}-E_k\mathbf U^{(\ell+1)})^TA u_j=0,\ \ \ \ j=1,...,k.
\end{eqnarray}
For the following proof, we define three diagonal matrices
$D={\rm diag}(\alpha_1,...,\alpha_k)$, $\Lambda={\rm diag}(\lambda_1,...,\lambda_k)$
and $\Lambda^{(\ell+1)}={\rm diag}(\lambda_1^{(\ell+1)},...,\lambda_k^{(\ell+1)})$.

Combining (\ref{Linear_Equation}), (\ref{Equality_9}),
(\ref{Inequality_Spectral_Projection}), (\ref{Equality_10}) and (\ref{Equality_11})
leads to the following estimate
\begin{eqnarray}\label{Inequality_12}
&&\sum_{i=1}^k\|u_i-E_{m,k}^{(\ell+1)}u_i\|_A^2 = \sum_{i=1}^k
\|\alpha_i u_i^{(\ell+1)}-E_k(\alpha_i u_i^{(\ell+1)})\|_A^2\nonumber\\
&=&{\rm trace}\left((\mathbf U^{(\ell+1)}D-E_k\mathbf U^{(\ell+1)}D)^T
A(\mathbf U^{(\ell+1)}D-E_k\mathbf U^{(\ell+1)}D)\right)\nonumber\\
&=&{\rm trace}\left(D^T(\mathbf U^{(\ell+1)}-E_k\mathbf U^{(\ell+1)})^T
A(\mathbf U^{(\ell+1)}-\pi_k\widetilde{\mathbf U}^{(\ell+1)}Q^{-1}\Lambda^{-1}Q
\Lambda^{(\ell+1)})D\right)\nonumber\\
&=&{\rm trace}\left(D(\mathbf U^{(\ell+1)}-E_k\mathbf U^{(\ell+1)})^T
A(\mathbf U^{(\ell+1)}-\mathbf UQQ^{-1}\Lambda^{-1}Q\Lambda^{(\ell+1)})D\right)\nonumber\\
&=&{\rm trace}\left(D(\mathbf U^{(\ell+1)}-E_k\mathbf U^{(\ell+1)})^T
(\widetilde{\mathbf U}^{(\ell+1)}\Lambda^{(\ell+1)}
-\mathbf U\Lambda QQ^{-1}\Lambda^{-1}Q\Lambda^{(\ell+1)})D\right)\nonumber\\
&=&{\rm trace}\left(D(\mathbf U^{(\ell+1)}-E_k\mathbf U^{(\ell+1)})^T
(\widetilde{\mathbf U}^{(\ell+1)}\Lambda^{(\ell+1)}
-\mathbf UQ\Lambda^{(\ell+1)})D\right)\nonumber\\
&=&{\rm trace}\left(D(\mathbf U^{(\ell+1)}-E_k\mathbf U^{(\ell+1)})^T
(\widetilde{\mathbf U}^{(\ell+1)}\Lambda^{(\ell+1)}
-\pi_k\widetilde{\mathbf U}^{(\ell+1)}\Lambda^{(\ell+1)})D\right)\nonumber\\
&=&\sum_{i=1}^k\alpha_i^2\lambda_i^{(\ell+1)}(\widetilde u_i^{(\ell+1)}
-\pi_k\widetilde u_i^{(\ell+1)},u_i^{(\ell+1)}-E_ku_i^{(\ell+1)})\nonumber\\
&\leq& \sum_{i=1}^k\alpha_i\frac{\lambda_i^{(\ell+1)}}{\sqrt{\lambda_{k+1}}}
\Big\|\widetilde u_i^{(\ell+1)}-\pi_k \widetilde u_i^{(\ell+1)}\Big\|_2
\Big\|\alpha_iu_i^{(\ell+1)}-E_k(\alpha_iu_i^{(\ell+1)})\Big\|_A.
\end{eqnarray}
From (\ref{Inequality_13}), (\ref{Equality_10_2}) and (\ref{Inequality_12}), we have
\begin{eqnarray*}
&&\sum_{i=1}^k\|u_i-E_{m,k}^{(\ell+1)}u_i\|_A^2=\sum_{i=1}^k
\|\alpha_i u_i^{(\ell+1)}-E_k(\alpha_i u_i^{(\ell+1)})\|_A^2\nonumber\\
&\leq& \sum_{i=1}^k\alpha_i^2\frac{\lambda_i^{(\ell+1)}}{\lambda_{k+1}}
\Big\|\sqrt{\lambda_i^{(\ell+1)}}\widetilde u_i^{(\ell+1)}
-\pi_k\big(\sqrt{\lambda_i^{(\ell+1)}}\widetilde u_i^{(\ell+1)}\big)\Big\|_2^2\nonumber\\
&\leq&\frac{\lambda_k^{(\ell+1)}}{\lambda_{k+1}} \sum_{i=1}^k
\Big\|\sqrt{\lambda_i^{(\ell+1)}}\widetilde u_i^{(\ell+1)}
-\pi_k \big(\sqrt{\lambda_i^{(\ell+1)}}\widetilde u_i^{(\ell+1)}\big)\Big\|_2^2\nonumber\\
&=&\frac{\lambda_k^{(\ell+1)}}{\lambda_{k+1}}\sum_{i=1}^k
\|\sqrt{\lambda_i}u_i-\tilde\pi_{m,k}^{(\ell+1)}
\big(\sqrt{\lambda_i}u_i\big)\|_2^2\leq \frac{\lambda_k^{(\ell+1)}}{\lambda_{k+1}}\sum_{i=1}^k
\lambda_i\|u_i-\tilde E_{m,k}^{(\ell+1)}u_i\|_2^2\nonumber\\
&\leq&\frac{\lambda_k^{(\ell+1)}}{\lambda_{k+1}} \sum_{i=1}^k
\lambda_i\Big(1+\frac{\mu_{k+1}^{(\ell+1)}\eta_{\mathcal K_k^{(\ell+1)}}^2}{\big(\delta_{k,i}^{(\ell+1)}\big)^2}\Big)
\eta_{\mathcal K_k^{(\ell+1)},k,i}^2\|u_i-E_{m,k}^{(\ell)}u_i\|_A^2\nonumber\\
&\leq&\frac{\lambda_k^{(\ell+1)}\lambda_k}{\lambda_{k+1}}
\Big(1+\frac{\mu_{k+1}^{(\ell+1)}\eta_{\mathcal K_k^{(\ell+1)}}^2}{\big(\delta_{k,k}^{(\ell+1)}\big)^2}\Big)
\eta_{\mathcal K_k^{(\ell+1)},k,k}^2\sum_{i=1}^k\|u_i-E_{m,k}^{(\ell)}u_i\|_A^2.
\end{eqnarray*}
This is the desired result (\ref{Error_Estimate_Inverse}) and the proof is complete.
\end{proof}
\begin{remark}
In the convergence result (\ref{Error_Estimate_Inverse}),
the term $\sqrt{\lambda_k/\lambda_{k+1}}$ comes from the inverse power iteration.
Different from the normal inverse power method, there exists the
term $\sqrt{\lambda_k^{(\ell+1)}}\eta_{\mathcal K_k^{(\ell+1)},k,k}$
which depends on the subspace $\mathcal K$. We can accelerate the inverse power iteration largely if
the subspace $\mathcal K$ can make the term
$\sqrt{\lambda_k^{(\ell+1)}}\eta_{\mathcal K_k^{(\ell+1)},k,k}$ be small (less than $1$).
\end{remark}
\begin{remark}
In this paper, we are only concerned with the error estimates for the eigenvector
approximation since the error estimates for the eigenvalue approximation
can be easily deduced from the following error expansion
\begin{eqnarray*}\label{rayexpan}
0\leq \widehat{\lambda}_i-\lambda_i
=\frac{\big(A(u_i-\psi),u_i-\psi\big)}{(\psi,\psi)}-\lambda_i
\frac{\big(u_i-\psi,u_i-\psi\big)}{(\psi,\psi)}\leq \frac{\|u_i-\psi\|_A^2}{\|\psi\|_2^2},
\end{eqnarray*}
where $\psi$ is the eigenvector approximation for the exact eigenvector $u_i$ and
\begin{eqnarray*}
\widehat{\lambda}_i=\frac{(A\psi,\psi)}{(\psi,\psi)}.
\end{eqnarray*}
\end{remark}
It is obvious that the parallel computing method can be used for Step 2 of Algorithm \ref{Algorithm_1}
since each linear equation can be solved independently. Furthermore, in order to design a complete parallel
scheme for eigenvalue problems, we give another version of the inverse power method for only one (may be not the
smallest one) eigenpair.

We start from an eigenvector approximation $u_i^{(\ell)}$ which is closest
to an exact eigenvector denoted by $u$. Then the new version of the inverse power iteration on a subspace
can be defined as follows:
\begin{algorithm}\label{Algorithm_2} {\bf Inverse power method on a subspace for one eigenvector}\\
For given eigenvector approximation $u_i^{(\ell)}$, do the following two steps
\begin{enumerate}
\item Define the subspace $\mathcal K^{(\ell+1)}
:=\mathcal K+{\rm span}\{u_i^{(\ell)}\}$
and solve the following eigenvalue problem:
Find $\widetilde u_i^{(\ell+1)}\in \mathcal K^{(\ell+1)}$ and $\lambda_i^{(\ell+1)}\in\mathbb R$
such that $\|\widetilde u_i^{(\ell+1)}\|_A=1$ and
\begin{eqnarray}\label{Subspace_Eigen_Problem_2}
(A \widetilde u_i^{(\ell+1)}, v) = \lambda_i^{(\ell+1)}(\widetilde u_i^{(\ell+1)},v),
\ \ \ \ \forall v\in \mathcal K^{(\ell+1)}.
\end{eqnarray}
Solve this eigenvalue problem to obtain a new eigenvector approximation
$\widetilde u_i^{(\ell+1)}$ which has the biggest orthogonal projection in the direction of $u_i^{(\ell)}$.
\item Solve the following linear equation:
\begin{eqnarray}\label{Linear_Equation_2}
A u_i^{(\ell+1)}=\lambda_i^{(\ell+1)}\widetilde u_i^{(\ell+1)}.
\end{eqnarray}
\end{enumerate}
We obtain the new eigenvector approximation $u_i^{(\ell+1)}$ as the output.
\end{algorithm}
We define spectral projection $E_{m}^{(i,\ell)}:\mathbb R^n\mapsto {\rm span}\{u_i^{(\ell)}\}$
corresponding to the eigenvector approximation $u_i^{(\ell)}$ as follows
\begin{eqnarray*}
(E_{m}^{(i,\ell)}w,u_i^{(\ell)})_A = (w, u_i^{(\ell)})_A,
\ \ \ \ \forall w\in \mathbb R^n.
\end{eqnarray*}
Then the spectral projections  $E_{m}^{(i,\ell+1)}$, $\tilde E_{m}^{(i,\ell+1)}$ and $E$ can also be defined
corresponding to eigenvectors $u_i^{(\ell+1)}$, $\widetilde u_i^{(\ell+1)}$ and $u$, respectively.
Based on Theorem \ref{Error_Estimate_Theorem}, we give the following error estimate for Algorithm \ref{Algorithm_2}.
\begin{theorem}\label{Theorem_Error_Estimate_Inverse_2}
There exists an exact eigenvector $u$ such that the resultant eigenvector approximation
$u_i^{(\ell+1)}$ has the following error estimate
\begin{eqnarray}\label{Error_Estimate_Inverse_2}
&&\|u-E_{m}^{(i,\ell+1)}u\|_A\leq \theta_{\mathcal K^{(\ell+1)}}\sqrt{\frac{\lambda\lambda_i^{(\ell+1)}}{\lambda_1}}
\eta_{\mathcal K^{(\ell+1)},i}\|u-E_{m}^{(i,\ell)}u\|_A,
\end{eqnarray}
where $\theta_{\mathcal K^{(\ell+1)}}$ and $\eta_{\mathcal K^{(\ell+1)},i}$ are defined as follows
\begin{eqnarray}
\theta_{\mathcal K^{(\ell+1)}}:=\sqrt{1+\frac{\mu_1^{(\ell+1)}\eta_{\mathcal K^{(\ell+1)}}^2}{\big(\delta_i^{(\ell+1)}\big)^2}},\ \ \ \ \ \
\eta_{\mathcal K^{(\ell+1)},i} := \Big(1+\frac{1}{\delta_{i}^{(\ell+1)}}\Big)\eta_{\mathcal K^{(\ell+1)}}.\label{Definition_Inverse_1_2}
\end{eqnarray}
with
\begin{eqnarray}
\delta_i^{(\ell+1)}:=\min_{j\neq i}|\mu_j^{(\ell+1)}-\mu|,\ \ \ {\rm and}\ \ \
\mu_j^{(\ell+1)} = \frac{1}{\lambda_j^{(\ell+1)}},\ \ \ j=1, ..., m.
\end{eqnarray}
\end{theorem}
\begin{proof}
From Theorem \ref{Error_Estimate_Theorem}, there exists an exact eigenvector $u$ closest to the eigenvector
approximation $u_i^{(\ell)}$ such that following error estimates for the eigenvector
$\widetilde u_i^{(\ell+1)}$ hold
\begin{eqnarray}
\|u-\tilde E_{m}^{(i,\ell+1)}u\|_A
&\leq& \sqrt{1+\frac{\mu_1^{(\ell+1)}\eta_{\mathcal K^{(\ell+1)}}^2}{\big(\delta_i^{(\ell+1)}\big)^2}}
\|(I-\mathcal P_{\mathcal K^{(\ell+1)}})u\|_A\nonumber\\
&\leq&\sqrt{1+\frac{\mu_1^{(\ell+1)}\eta_{\mathcal K^{(\ell+1)}}^2}{\big(\delta_i^{(\ell+1)}\big)^2}}\|u-E_{m}^{(i,\ell)}u\|_A,
\end{eqnarray}
and
\begin{eqnarray}\label{Inequality_13_2}
\|u-\tilde E_{m}^{(i,\ell+1)}u\|_2 &\leq& \eta_{\mathcal K^{(\ell+1)},i}
\|u-\tilde E_{m}^{(i,\ell+1)}u\|_A\nonumber\\
&\leq&\eta_{\mathcal K^{(\ell+1)},i}
\sqrt{1+\frac{\mu_1^{(\ell+1)}\eta_{\mathcal K^{(\ell+1)}}^2}{\big(\delta_i^{(\ell+1)}\big)^2}}
\|u-E_{m}^{(i,\ell)}u\|_A.
\end{eqnarray}
Let $\alpha_i=1/\|u_i^{(\ell+1)}\|_A$. Similarly, from (\ref{Subspace_Eigen_Problem_2}) and (\ref{Linear_Equation_2}),
we have the following inequality
\begin{eqnarray*}
1= \lambda_i^{(\ell+1)}(\widetilde u_i^{(\ell+1)},\widetilde u_i^{(\ell+1)})
= (A u_i^{(\ell+1)},\widetilde u_i^{(\ell+1)})\leq \|u_i^{(\ell+1)}\|_A\|\widetilde u_i^{(\ell+1)}\|_A=\frac{1}{\alpha_i},
\end{eqnarray*}
which leads to the estimate $\alpha_i\leq 1$.
For the analysis, we define the $L^2$-projections $\pi$ and $\tilde\pi_{m}^{(i,\ell+1)}$
corresponding to the spaces ${\rm span}\{u\}$ and
${\rm span}\{ \widetilde u_i^{(\ell+1)}\}$, respectively.
Then from $\|u\|_A=\|\alpha_iu_i^{(\ell+1)}\|_A=1$ and
$\|\sqrt{\lambda}u\|_2=\|\sqrt{\lambda_i^{(\ell+1)}}\widetilde u_i^{(\ell+1)}\|_2=1$, we have following equalities
\begin{eqnarray}\label{Equality_9_2}
\|u-E_{m}^{(i,\ell+1)}u\|_A= \|\alpha_i u_i^{(\ell+1)}-E(\alpha_i u_i^{(\ell+1)})\|_A
\end{eqnarray}
and
\begin{eqnarray}\label{Equality_10_2_2}
\Big\|\sqrt{\lambda_i^{(\ell+1)}}\widetilde u_i^{(\ell+1)}
-\pi \big(\sqrt{\lambda_i^{(\ell+1)}}\widetilde u_i^{(\ell+1)}\big)\Big\|_2
=\|\sqrt{\lambda}u-\tilde\pi_{m}^{(i,\ell+1)}\big(\sqrt{\lambda}u\big)\|_2.
\end{eqnarray}
From the definition of the eigenvalue, it is easy to know the following property holds
\begin{eqnarray}\label{Inequality_Spectral_Projection_2}
\frac{\|u_i^{(\ell+1)}-Eu_i^{(\ell+1)}\|_A^2}{\|u_i^{(\ell+1)}-Eu_i^{(\ell+1)}\|_2^2}
\geq \lambda_1.
\end{eqnarray}
Combining (\ref{Linear_Equation_2}), (\ref{Equality_9_2}) and (\ref{Inequality_Spectral_Projection_2}) leads to the following estimate
\begin{eqnarray}\label{Inequality_12_2}
&&\|u-E_{m}^{(i,\ell+1)}u\|_A^2 = \|\alpha_i u_i^{(\ell+1)}-E(\alpha_i u_i^{(\ell+1)})\|_A^2\nonumber\\
&=& \Big(\alpha_i u_i^{(\ell+1)}-E(\alpha_i u_i^{(\ell+1)}), \alpha_i u_i^{(\ell+1)}-E(\alpha_i u_i^{(\ell+1)})\Big)_A\nonumber\\
&=& \Big(\alpha_i u_i^{(\ell+1)}-E(\alpha_i u_i^{(\ell+1)}), \alpha_i u_i^{(\ell+1)}-\frac{\lambda_i^{(\ell+1)}}{\lambda}\pi(\alpha_i
\widetilde u_i^{(\ell+1)})\Big)_A\nonumber\\
&=& \lambda_i^{(\ell+1)}\Big(\alpha_i u_i^{(\ell+1)}-E(\alpha_i u_i^{(\ell+1)}),
\alpha_i u_i^{(\ell+1)}-\pi(\alpha_i \widetilde u_i^{(\ell+1)})\Big)\nonumber\\
&\leq& \lambda_i^{(\ell+1)} \|\alpha_i \widetilde u_i^{(\ell+1)}-\pi(\alpha_i \widetilde u_i^{(\ell+1)})\|_2
\|\alpha_i u_i^{(\ell+1)}-E(\alpha_i u_i^{(\ell+1)})\|_2\nonumber\\
&=& \alpha_i\frac{\lambda_i^{(\ell+1)}}{\sqrt{\lambda_1}} \| \widetilde u_i^{(\ell+1)}-\pi \widetilde u_i^{(\ell+1)}\|_2
\|\alpha_i u_i^{(\ell+1)}-E(\alpha_i u_i^{(\ell+1)})\|_A\nonumber\\
&=& \alpha_i\frac{\lambda_i^{(\ell+1)}}{\sqrt{\lambda_1}}
\|\widetilde u_i^{(\ell+1)}-\pi \widetilde u_i^{(\ell+1)}\|_2\|u-E_{m}^{(i,\ell+1)}u\|_A.
\end{eqnarray}
From (\ref{Inequality_13_2}), (\ref{Equality_10_2_2}) and (\ref{Inequality_12_2}), we have
\begin{eqnarray*}
&&\|u-E_{m}^{(i,\ell+1)}u\|_A\leq \alpha_i\frac{\lambda_i^{(\ell+1)}}{\sqrt{\lambda_1}}
\| \widetilde u_i^{(\ell+1)}-\pi \widetilde u_i^{(\ell+1)}\|_2\nonumber\\
&\leq&\alpha_i\sqrt{\frac{\lambda_i^{(\ell+1)}}{\lambda_1}}
\|\sqrt{\lambda_i^{(\ell+1)}}\widetilde u_i^{(\ell+1)}-\pi(\sqrt{\lambda_i^{(\ell+1)}}\widetilde u_i^{(\ell+1)})\|_2\nonumber\\
&=& \alpha_i\sqrt{\frac{\lambda_i^{(\ell+1)}}{\lambda_1}}\|\sqrt{\lambda} u-\tilde\pi_{m}^{(i,\ell+1)}(\sqrt{\lambda}u)\|_2
\leq \alpha_i\sqrt{\frac{\lambda\lambda_i^{(\ell+1)}}{\lambda_1}}\|u-E_m^{(i,\ell)}u\|_2\nonumber\\
&\leq&\eta_{\mathcal K^{(\ell+1)},i}\sqrt{\frac{\lambda\lambda_i^{(\ell+1)}}{\lambda_1}}
\sqrt{1+\frac{\mu_1^{(\ell+1)}\eta_{\mathcal K^{(\ell+1)}}^2}{\big(\delta_i^{(\ell+1)}\big)^2}}\|u-E_m^{(i,\ell)}u\|_A.
\end{eqnarray*}
This is the desired result (\ref{Error_Estimate_Inverse_2}) and the proof is complete.
\end{proof}
From the estimate (\ref{Error_Estimate_Inverse_2}), in order to guarantee the convergence of Algorithm \ref{Algorithm_2},
we need to choose the subspace $\mathcal K$ properly such that the term
$\sqrt{\lambda\lambda_i^{(\ell+1)}/\lambda_1}\eta_{\mathcal K^{(\ell+1)},i}$
is small which is stricter than the condition $\sqrt{\lambda_k^{(\ell+1)}}\eta_{\mathcal K_k^{(\ell+1)},k,k}$
is small when $\lambda>\lambda_1$. But we can implement Algorithm \ref{Algorithm_2} in parallel for
different eigenpairs which is the most important advantage of this algorithm.

\section{Geometric multigrid method for eigenvalue problem}
In this section, we discuss a type of geometric multigrid (GMG) method  for
the standard elliptic eigenvalue problem \cite{LinXie,Xie_JCP,Xie_IMA}.
Here, the standard notation for Sobolev spaces $H^s(\Omega)$
and their associated norms and semi-norms \cite{Adams} will be used.
We denote $H_0^1(\Omega)=\{v\in H^1(\Omega):\ v|_{\partial\Omega}=0\}$,
where $v|_{\partial\Omega}=0$ is in the sense of trace.
The letter $C$ (with or without subscripts) denotes a generic positive constant
which may be different at its different occurrences in this section.

The concerned eigenvalue problem in this section is defined as follows: Find $(\lambda,u)$ such that
\begin{equation}\label{LaplaceEigenProblem}
\left\{
\begin{array}{rcl}
-\Delta u&=&\lambda u,\ \ \  {\rm in}\ \Omega,\\
u&=&0,\ \ \ \ \  {\rm on}\ \partial\Omega.
\end{array}
\right.
\end{equation}

In order to use the finite element method to solve
the eigenvalue problem (\ref{LaplaceEigenProblem}), we need to define
the corresponding variational form as follows:
Find $(\lambda, u )\in \mathbb{R}\times V$ such that $a(u,u)=1$ and 
\begin{eqnarray}\label{weak_eigenvalue_problem}
a(u,v)&=&\lambda b(u,v),\quad \forall v\in V,
\end{eqnarray}
where $V:=H_0^1(\Omega)$ and
\begin{equation}\label{inner_product_a_b}
a(u,v)=\int_{\Omega}\nabla u\cdot\nabla vd\Omega,
 \ \ \ \  \ \ b(u,v) = \int_{\Omega}uv d\Omega.
\end{equation}
The norms $\|\cdot\|_a$ and $\|\cdot\|_b$ are defined as
\begin{eqnarray*}
\|v\|_a=\sqrt{a(v,v)}\ \ \ \ \ {\rm and}\ \ \ \ \ \|v\|_b=\sqrt{b(v,v)}.
\end{eqnarray*}
Now, we introduce the finite element method for the eigenvalue problem
(\ref{weak_eigenvalue_problem}). First we decompose the computing domain
$\Omega\subset \mathbb{R}^d\ (d=2,3)$
into shape-regular triangles or rectangles for $d=2$ (tetrahedrons or
hexahedrons for $d=3$) to produce the mesh $\mathcal{T}_h$ (cf. \cite{BrennerScott,Ciarlet}).
The diameter of a cell $K\in\mathcal{T}_h$ is denoted by $h_K$ and
the mesh size $h$ describes  the maximum diameter of all cells
$K\in\mathcal{T}_h$. Based on the mesh $\mathcal{T}_h$, we
construct the linear finite element space  $V_h \subset V$ as follows:
\begin{equation}\label{linear_fe_space}
V_h = \big\{ v_h \in C(\Omega)\ \big|\ v_h|_{K} \in \mathcal{P}_1,
\ \ \forall K \in \mathcal{T}_h\big\}\cap H_0^1(\Omega),
\end{equation}
where $\mathcal{P}_1$ denotes the space of polynomials of degree at most $1$.

The standard finite element scheme for the eigenvalue
 problem (\ref{weak_eigenvalue_problem}) can be defined as follows:
Find $(\lambda_h, u_h)\in \mathbb{R}\times V_h$
such that $a(u_h,u_h)=1$ and
\begin{eqnarray}\label{Weak_Eigenvalue_Discrete}
a(u_h,v_h)
&=&\lambda_h b(u_h,v_h),\quad\ \  \ \forall v_h\in V_h.
\end{eqnarray}

Based on the basis system, the discrete eigenvalue problem (\ref{Weak_Eigenvalue_Discrete})
can be transformed to the following general algebraic eigenvalue problem:
Find $(\lambda,u)\in \mathbb R\times \mathbb R^n$ such that
\begin{eqnarray}\label{Eigenvalue_Equation_FEM}
Au&=&\lambda M u,
\end{eqnarray}
where $n:={\rm dim}(V_h)$, $A$ and $M$ denote the stiff and mass matrices, respectively,
corresponding to the finite element space $V_h$.

For simplicity of description in this and next sections, we only consider the applications of
Algorithm \ref{Algorithm_1}. It is not difficult to understand that Algorithm \ref{Algorithm_2}
can be used similarly. Theorem \ref{Theorem_Error_Estimate_Inverse} can give the understanding
and also a new proof for the GMG method \cite{Xie_JCP} for the eigenvalue
problem (\ref{Weak_Eigenvalue_Discrete}). In this GMG method, the basic space $\mathcal K$
is chosen as the low dimensional
finite element space $V_H$ which is defined on the coarse mesh $\mathcal T_H$ and
it is obvious that $\mathcal K:=V_H\subset V_h$. In order to use
Algorithm \ref{Algorithm_1}, we only need to define the $L^2$ inner product
in (\ref{Subspace_Eigen_Problem}) with the mass matrix $M$ by the following way
\begin{eqnarray*}
(v,w)= v^TMw,\ \ \ \ \forall v\in\mathbb R^n\  {\rm and}\ \forall w\in \mathbb R^n
\end{eqnarray*}
and the linear equations (\ref{Linear_Equation}) are replaced by the
following standard linear equations
\begin{eqnarray}\label{Linear_Equation_M}
Au_i^{(\ell+1)}=\lambda_i^{(\ell+1)}M\widetilde u_i^{(\ell+1)},\ \ \ i=1,...,k,
\end{eqnarray}
which can be solved by the well-known GMG method for boundary value problems
\cite{Bramble,BrambleZhang,BrennerScott,Hackbusch_Book,McCormick,RugeStuben,
Shaidurov,Stuben,TrottenbergOosterleeSchuller,Xu1992Iterative}.

With the standard results from the finite element
theory \cite{BrennerScott,Ciarlet}, we can give the following estimates
for the quantities appeared in Theorem \ref{Theorem_Error_Estimate_Inverse}.
Combining (\ref{Definition_Inverse_1}) and the well-known Aubin-Nitsche result
$\eta_{\mathcal K_k^{(\ell+1)}}\leq \eta_{\mathcal K}\leq CH$ (cf. \cite{BrennerScott,Ciarlet,Xie_JCP})
leads to the following estimate
\begin{eqnarray}\label{Inequality_14}
\eta_{\mathcal K_k^{(\ell+1)},k,k}=\Big(1+\frac{\mu_{k+1}^{(\ell+1)}}{\delta_{k,k}^{(\ell+1)}}\Big)
\eta_{\mathcal K_k^{(\ell+1)}} \leq
\Big(1+\frac{\mu_{k+1}}{\delta_{k,k}^{(\ell+1)}}\Big)\eta_\mathcal K
=\Big(1+\frac{1}{\delta_{k,k}^{(\ell+1)}\lambda_{k+1}}\Big)\eta_{\mathcal K}\leq CH,
\end{eqnarray}
where we use the property $\mu_{k+1}^{(\ell+1)}\leq \mu_{k+1}=1/\lambda_{k+1}$ and
the constant $C$ depends on the eigenvalue gap $\delta_{k,k}^{(\ell+1)}$ and the eigenvalue $\lambda_{k+1}$.
From Theorem \ref{Theorem_Error_Estimate_Inverse} and (\ref{Inequality_14}), the following convergence result holds
\begin{eqnarray}\label{Error_Estimate_Inverse_GMG}
\Big(\sum_{i=1}^k\|u_i-E_{m,k}^{(\ell+1)}u_i\|_A^2\Big)^{1/2}
\leq C\sqrt{\lambda_k^{(\ell+1)}}H\Big(\sum_{i=1}^k\|u_i-E_{m,k}^{(\ell)}u_i\|_A^2\Big)^{1/2}.
\end{eqnarray}
From (\ref{Error_Estimate_Inverse_GMG}), in order to produce the uniform convergence,
we only need to choose the size $H$ of the coarse mesh $\mathcal T_H$ small enough
such that the following condition holds
\begin{eqnarray}\label{Mesh_Condition}
C\sqrt{\lambda_k^{(\ell+1)}}H<1.
\end{eqnarray}

Since the definition $\mathcal K=V_H$ results in the sparse matrices for the eigenvalue
problem (\ref{Subspace_Eigen_Problem}), the required memory for Algorithm \ref{Algorithm_1} is almost optimal and
less than mostly existed eigenvalue solvers. For more details, please refer to papers \cite{LinXie,Xie_JCP,Xie_IMA}
and the numerical examples provided there. Different from the existed GMG methods for eigenvalue problems \cite{BiYang,BrandtMcCormickRuge,CaiMandelMcCormick,ChenHeLiXie,Hackbusch,Hackbusch_Book,Shaidurov},
our method only need to solve standard linear elliptic boundary value problems (\ref{Linear_Equation_M})
and any efficient linear solvers can be used without any modification. The GMG methods in \cite{BrandtMcCormickRuge,CaiMandelMcCormick,Hackbusch,Hackbusch_Book,Shaidurov}
are designed based on the shift-inverse power method for eigenvalue problem and GMG solver
is used as an inner iteration with modifications since there appear singular
or nearly singular linear equations. Furthermore, since the scale of
eigenvalue problem (\ref{Subspace_Eigen_Problem}) that we need to solve is very small,
the choice of the corresponding eigenvalue solver is very free.

\section{Algebraic multigrid method for eigenvalue problem}
Based on Algorithm \ref{Algorithm_1} and the corresponding convergence result in
Theorem \ref{Theorem_Error_Estimate_Inverse} in Section 4, similarly to the idea
presented in the previous section, we can design and analyze a type of algebraic multigrid (AMG)
method for the eigenvalue problem (\ref{Eigenvalue_Equation_0}).
If we can find a suitable low dimensional subspace $\mathcal K$ by some type of coarsening step,
a type of inverse power method with a fast convergence rate can be designed based on Algorithm \ref{Algorithm_1}.
Inspired by the AMG method for linear equations \cite{RugeStuben,Stuben,TrottenbergOosterleeSchuller,XuZikatanov},
the natural low dimensional subspace can be chosen as the coarse space in AMG method and the AMG method
is also an efficient solver for linear equations (\ref{Linear_Equation}).

In this section, we set $n_c={\rm dim}(\mathcal K)<n$ as the dimension of the subspace $\mathcal K$.
First, let us consider a special case that the subspace $\mathcal K:={\rm span}\{u_j\}_{j=1}^{n_c}$
which is constituted by the eigenvectors corresponding to the smallest $n_c$ eigenvalues of $A$.
In the following analysis, we assume any vector $g\in \mathbb R^n$ has the expansion $g=\sum_{j=1}^n\alpha_ju_j$.
Then  the following inequalities hold
and
\begin{eqnarray}\label{Eta_AMG}
\eta_{\mathcal K_k^{(\ell+1)}}&:=&\sup_{\|g\|_2=1}\|(I-\mathcal P_{\mathcal K_k^{(\ell+1)}})Tg\|_A
\leq \sup_{\|g\|_2=1}\Big\|(I-\mathcal P_\mathcal K)\sum_{j=1}^n\mu_j\alpha_ju_j\Big\|_A\nonumber\\
&=&\sup_{\|g\|_2=1}\Big\|\sum_{j=n_c+1}^n\mu_j\alpha_ju_j\Big\|_A
=\sup_{\|g\|_2=1}\Big\|\sum_{j=n_c+1}^n\sqrt{\lambda_j}\mu_j\alpha_ju_j\Big\|_2\nonumber\\
&=&\sup_{\|g\|_2=1}\Big\|\sum_{j=n_c+1}^n\sqrt{\mu_j}\alpha_ju_j\Big\|_2\leq \sqrt{\mu_{n_c+1}}
\sup_{\|g\|_2=1}\Big\|\sum_{j=n_c+1}^n\alpha_ju_j\Big\|_2
\leq \sqrt{\mu_{n_c+1}}.
\end{eqnarray}
In (\ref{Definition_Inverse_2}) and (\ref{Definition_Inverse_1}),
since $\mathcal K\subset\mathcal K_k^{(\ell+1)}$, (\ref{Eta_AMG}) and the property
$\mu_{k+1}^{(\ell+1)}\leq \mu_{k+1}=1/\lambda_{k+1}$, we have the following estimates
\begin{eqnarray}
&&\theta_{\mathcal K_k^{(\ell+1)}}\leq \sqrt{1+\frac{1}{\lambda_{k+1}\lambda_{n_c+1}\big(\delta_{k,k}^{(\ell+1)}\big)^2}},\label{Estimates_AMG_1}\\
&&\eta_{\mathcal K_k^{(\ell+1)},k,k}=\Big(1+\frac{\mu_{k+1}^{(\ell+1)}}{\delta_{k,k}^{(\ell+1)}}\Big)
\eta_{\mathcal K_k^{(\ell+1)}} \leq
\Big(1+\frac{1}{\lambda_{k+1}\delta_{k,k}^{(\ell+1)}}\Big)\sqrt{\frac{1}{\lambda_{n_c+1}}}.\label{Estimates_AMG_2}
\end{eqnarray}
Then from (\ref{Error_Estimate_Inverse}), (\ref{Estimates_AMG_1}) and (\ref{Estimates_AMG_2}),
the convergence rate of Algorithm \ref{Algorithm_1} has the following estimate
\begin{eqnarray}\label{Error_Estimate_Inverse_AMG}
&&\Big(\sum_{i=1}^k\|u_i-E_{m,k}^{(\ell+1)}u_i\|_A^2\Big)^{1/2}\nonumber\\
&&\hskip-1.4cm\leq\sqrt{1+\frac{1}{\lambda_{k+1}\lambda_{n_c+1}
\big(\delta_{k,k}^{(\ell+1)}\big)^2}}\sqrt{\frac{\lambda_k}{\lambda_{k+1}}}
\Big(1+\frac{1}{\lambda_{k+1}\delta_{k,k}^{(\ell+1)}}\Big)\sqrt{\frac{\lambda_k^{(\ell+1)}}{\lambda_{n_c+1}}}
\Big(\sum_{i=1}^k\|u_i-E_{m,k}^{(\ell)}u_i\|_A^2\Big)^{1/2}.
\end{eqnarray}
From (\ref{Error_Estimate_Inverse_AMG}), the convergence speed can be improved from $\sqrt{\lambda_k/\lambda_{k+1}}$
if $\lambda_k^{(\ell+1)}<\lambda_{n_c+1}$ which only need $n_c+1>k$ and $\lambda_k^{(\ell+1)}$
has a coarse accuracy.

If the algebraic eigenvalue problem (\ref{Eigenvalue_Equation_0}) is produced by the discretization of
the partial differential operator eigenvalue problem (\ref{LaplaceEigenProblem}) with the finite element method,
the Weyl's law \cite{LiYau,Weyl} tells us $\lambda_j$ has the following asymptotic estimate
\begin{eqnarray*}
\lambda_{j} \approx \left(\frac{j}{|\Omega|}\right)^{2/d},\ \ \ \forall j\in \mathbb N,
\end{eqnarray*}
which leads to the following estimate
\begin{eqnarray}\label{Error_Estimate_Inverse_AMG_2}
\Big(\sum_{i=1}^k\|u_i-E_{m,k}^{(\ell+1)}u_i\|_A^2\Big)^{1/2}
&\leq&C\sqrt{\frac{\lambda_k}{\lambda_{k+1}}}\Big(\frac{k}{n_c+1}\Big)^{1/d}
\Big(\sum_{i=1}^k\|u_i-E_{m,k}^{(\ell)}u_i\|_A^2\Big)^{1/2},
\end{eqnarray}
where $d$ is the dimension of computing domain $\Omega$ and $|\Omega|$ denotes the volume of $\Omega$.
The estimate (\ref{Error_Estimate_Inverse_AMG_2}) means that we can improve the convergence rate if
$k<n_c+1$, i.e., the dimension of the subspace $\mathcal K$ is larger than the number of desired eigenvalues.

Since the eigenvectors of the matrix $A$ are more expensive to compute, the practical value of above estimates is limited. But they
provide a useful guidance to design practical AMG method for eigenvalue problems.  From (\ref{Error_Estimate_Inverse})
 and (\ref{Eta_AMG}), the first criterion for constructing the subspace $\mathcal K$
is that it can approximate the eigenvectors corresponding to small eigenvalues \cite{XuZikatanov}.
In order to reduce the computation, the second criterion is that we can use the sparse representation of the subspace $\mathcal K$.
Fortunately, a suitable coarse space of the AMG method for the linear equations satisfies these two criterions.
Thus, we can use the usual coarsening scheme to produce the low dimensional subspace $\mathcal K$ for Algorithm \ref{Algorithm_1}
which can be called AMG method for eigenvalue problems. For more information, please refer to \cite{HanHeXieYou}.

\section{Concluding remarks}
In this paper, we give the energy error estimate of the subspace projection method for eigenvalue problems.
Furthermore, the relation between error estimates in $L^2$-norm and energy norm is also provided.
Based on the energy error estimate and the relation, a new type of inverse power method based on
the subspace projection method is proposed and the convergence analysis is also presented.
Then we discuss the geometric and algebraic multigrid methods for eigenvalue problems
based on the derived convergence result for the proposed inverse power method on the special subspace.

These analysis and discussion give us a new understanding of the subspace projection method and provide
a new idea to design the multigrid method for eigenvalue problems. We would like to point out that the
most important aim of this paper is to present the idea and understanding of the application of
the coarse subspace, which can be produced by the coarse mesh in GMG and coarsening technique in AMG,
to eigenvalue problems. Of course, the idea or tool here can be coupled with
other techniques such as shift and inverse, polynomial filtering, restarting (cf. \cite{BaiDemmelDongarraRuheVorst,Yousef}).
These will be investigated in our future work.


\end{document}